\documentclass[graybox]{svmult}

\usepackage{mathptmx}       
\usepackage{helvet}         
\usepackage{courier}        
\usepackage{type1cm}        
\usepackage{makeidx}         
\usepackage{graphicx}        
\usepackage{multicol}        
\usepackage[bottom]{footmisc}

\usepackage{amsmath,amssymb,amscd}
\usepackage{bbm}

\makeindex             

\begin{document}

\newcommand{\burl}[1]{\textcolor{blue}{\url{#1}}}
\newcommand{\bemail}[1]{\email{\textcolor{blue}{#1}}}
\newcommand{\supp}{{\rm supp}}
\newcommand{\kot}[1]{ \frac{\sin \pi({#1}) }{\pi ({#1})} }
\newcommand{\kkot}[1]{ \frac{\sin \pi {#1} }{\pi {#1} } }
\newcommand{\hphi}{\widehat{\phi}}  
\newcommand{\Norm}[1]{\frac{#1}{\sqrt{N}}}
\newcommand\be{\begin{equation}}
\newcommand\ee{\end{equation}}
\newcommand\bea{\begin{eqnarray}}
\newcommand\eea{\end{eqnarray}}
\newcommand{\R}{\ensuremath{\mathbb{R}}}
\newcommand{\C}{\ensuremath{\mathbb{C}}}
\newcommand{\Z}{\ensuremath{\mathbb{Z}}}
\newcommand{\Q}{\mathbb{Q}}
\newcommand{\N}{\mathbb{N}}
\newcommand{\foh}{\frac{1}{2}}  
\newtheorem{thm}{Theorem}[section]
\newtheorem{conj}[thm]{Conjecture}
\newtheorem{cor}[thm]{Corollary}
\newtheorem{lem}[thm]{Lemma}
\newtheorem{prop}[thm]{Proposition}
\newtheorem{defi}[thm]{Definition}
\newtheorem{rek}[thm]{Remark}
\newcommand{\mat}[4]{\ensuremath{\left(\begin{array}{ll}#1 &#2 \\ #3 &#4%
\end{array}\right)}}
\newcommand{\ord}{\mathop{\mathrm{ord}}}
\newcommand{\ch}{\mathop{\mathrm{char}}}
\newcommand{\Supp}{\mathop{\mathrm{supp}}}
\newcommand{\Avg}{\mathop{\mathrm{Avg}}}
\renewcommand{\ch}{\mathop{\mathrm{ch}}}
\newcommand{\SL}{\mathrm{SL}}
\newcommand{\GL}{\mathrm{GL}}
\newcommand{\iso}{\cong}
\newcommand{\eps}{\epsilon}
\newcommand{\sgn}{\mathrm{sgn}}
\newcommand{\hfrak}{\mathfrak{h}}

\title*{Maass waveforms and low-lying zeros}
\author{Levent Alpoge, Nadine Amersi, Geoffrey Iyer, Oleg Lazarev, Steven J. Miller and Liyang Zhang}
\institute{Levent Alpoge \at Harvard University, Department of Mathematics, Harvard College, Cambridge, MA 02138  \hfill \\ \bemail{alpoge@college.harvard.edu} \and Nadine Amersi \at Department of Mathematics, University College London, London, WC1E 6BT  \hfill \\ \bemail{n.amersi@ucl.ac.uk} \and Geoffrey Iyer \at Department of Mathematics, UCLA, Los Angeles, CA 90095  \bemail{geoff.iyer@gmail.com} \and Oleg Lazarev \at  Department of Mathematics, Stanford University, Stanford, CA 94305  \hfill \\ \bemail{olazarev@stanford.edu} \and Steven J. Miller \at Department of Mathematics \& Statistics, Williams College, Williamstown, MA 01267  \hfill \\ \bemail{sjm1@williams.edu, Steven.Miller.MC.96@aya.yale.edu} \and Liyang Zhang \at Department of Mathematics, Yale University, New Haven, CT 06520  \hfill \\ \bemail{zhangliyangmath@gmail.com}
}

%
%
\maketitle

\abstract*{The Katz-Sarnak Density Conjecture states that the behavior of zeros of a family of $L$-functions near the central point (as the conductors tend to zero) agrees with the behavior of eigenvalues near 1 of a classical compact group (as the matrix size tends to infinity). Using the Petersson formula, Iwaniec, Luo and Sarnak proved that the behavior of zeros near the central point of holomorphic cusp forms agrees with the behavior of eigenvalues of orthogonal matrices for suitably restricted test functions $\phi$. We prove similar results for families of cuspidal Maass forms, the other natural family of ${\rm GL}_2/\mathbb{Q}$ $L$-functions. For suitable weight functions on the space of Maass forms, the limiting behavior agrees with the expected orthogonal group. We prove this for $\Supp(\widehat{\phi})\subseteq (-3/2, 3/2)$ when the level $N$ tends to infinity through the square-free numbers; if the level is fixed the support decreases to being contained in $(-1,1)$, though we still uniquely specify the symmetry type by computing the 2-level density.}

\abstract{The Katz-Sarnak Density Conjecture states that the behavior of zeros of a family of $L$-functions near the central point (as the conductors tend to zero) agrees with the behavior of eigenvalues near 1 of a classical compact group (as the matrix size tends to infinity). Using the Petersson formula, Iwaniec, Luo and Sarnak proved that the behavior of zeros near the central point of holomorphic cusp forms agrees with the behavior of eigenvalues of orthogonal matrices for suitably restricted test functions $\phi$. We prove similar results for families of cuspidal Maass forms, the other natural family of ${\rm GL}_2/\mathbb{Q}$ $L$-functions. For suitable weight functions on the space of Maass forms, the limiting behavior agrees with the expected orthogonal group. We prove this for $\Supp(\widehat{\phi})\subseteq (-3/2, 3/2)$ when the level $N$ tends to infinity through the square-free numbers; if the level is fixed the support decreases to being contained in $(-1,1)$, though we still uniquely specify the symmetry type by computing the 2-level density.}



\section{Introduction}

In this section we set the stage for our results by quickly reviewing previous work on zeros of $L$-functions, leading up to $n$-level correlations, densities and the conjectured correspondence with random matrix ensembles. As this is a vast field and the readership of this book is likely to have diverse backgrounds and interests, we discuss in some detail the history of the subject in order to put the present problems in context. We concentrate on some of the key theorems and statistics, and refer the reader to the extensive literature for more information. After this quick tour we describe the Katz-Sarnak conjectures for the behavior of low-lying zeros, and then in \S\ref{sec:mainresultsAAILMZ} we state our new results for families of Maass forms (the reader familiar with this field can skip this section and go straight to \S\ref{sec:mainresultsAAILMZ}). The analysis proceeds by using the Kuznetsov trace formula to convert sums over zeros to exponential sums over the primes. Similar sums have been extensively studied by Maier in many papers over the years (see for example \cite{EMaS,Ma,MaP,MaS1,MaS2,MaT}); it is a pleasure to dedicate this chapter to him on the occasion of his 60\textsuperscript{th} birthday.

\subsection{Zeros of $L$-Functions}

The Riemann zeta function $\zeta(s)$ is defined for ${\rm Re}(s) > 1$ by \be \zeta(s) \ := \ \sum_{n=1}^\infty \frac1{n^s} \ = \ \prod_{p\ {\rm prime}}\left(1 - \frac1{p^s}\right)^{-1}; \ee the Euler product expansion is equivalent to the Fundamental Theorem of Arithmetic on the unique factorization of integers into prime powers. Much can be gleaned in this regime. For example, looking at the limit as $s\to 1$ from above shows the sum of the reciprocals of the primes diverge (and with just a little work one gets $\sum_{p < x} 1/p \sim \log\log x$), and hence there are infinitely many prime. The true utility of this function, however, doesn't surface until we consider its meromorphic continuation $\xi(s)$ to the entire complex plane, where \be \xi(s) \ := \ \Gamma(s/2) \pi^{-s/2}\zeta(s) \ = \ \xi(1-s). \ee The product expansion shows $\xi(s)$ has no zeros for ${\rm Re}(s) > 1$, and from the functional equation the only zeros for ${\rm Re}(s) < 0$ are at the negative even integers. The remaining zeros all have real part between 0 and 1; the Riemann Hypothesis \cite{Rie} is the statement that these zeros all have real part of 1/2.

Ever since Riemann's classic paper, researchers have exploited the connections between zeros of $\zeta(s)$ (and later other $L$-functions) to arithmetically important problem to translate information about the zeros to results in number theory. For example, it can be shown that $\zeta(s)$ is never zero on the line ${\rm Re}(s) = 1$. This implies the Prime Number Theorem: the number of primes at most $x$, $\pi(x)$, is ${\rm Li}(x)$ plus a lower order term, where \be {\rm Li}(x) \ := \ \int_2^x \frac{dt}{\log t};\ee  excellent references for this and the subsequent results are \cite{Da, IK}.

Similar results about primes in arithmetic progressions modulo $m$ follow from analogous results about the distribution of zeros of Dirichlet $L$-functions $L(s,\chi) := \sum_n \chi(n) / n^s$, where $\chi$ ranges over all primitive characters modulo $m$. It is worth noting that to study primes congruent to $a$ modulo $m$ it is not enough to study \emph{one} specific Dirichlet $L$-function, but rather we need to understand the entire family coming from all characters of modulus $m$ in order to invoke orthogonality relations to extract information about our progression from averages of $\chi(m)$; this notion of family will be very important in our work later.

After determining main terms, it is natural to ask about the form of the lower order terms. While the Riemann Hypothesis (RH) implies that $\pi(x) = {\rm Li}(x) + O(x^{1/2}\log x)$, neither it nor its generalization to other $L$-functions (GRH) is powerful enough to explain how the distribution of primes modulo $m$ varies with the residue class, as these fluctuation are at the size of the errors from GRH. Chebyshev observed that there appeared to be more primes congruent to 3 modulo 4 than to 1 modulo 4. We now have an excellent theory (see \cite{RubSa}) that explains this phenomenon. A key ingredient is the Grand Simplicity Hypothesis, which asserts that the zeros of these $L$-functions are linearly independent over the rationals.

Assuming RH, the non-trivial zeros of $\zeta(s)$ all have real part equal to 1/2, and may thus be ordered on the line. It therefore makes sense to talk about spacings between adjacent zeros $\rho_j = 1/2 + i \gamma_j$, or better yet spacings between adjacent normalized zeros (where we have normalized so that the average spacing is 1). Recent work has shown powerful connections between these gaps and important arithmetic quantities. For example, we can obtain excellent bounds on the size of the class groups of imaginary quadratic fields through knowing the existence of $L$-functions with multiple zeros at the central point \cite{Go,GZ}, or knowing that a positive percentage of gaps between normalized zeros of the Riemann zeta function are at least a certain fixed fraction of the average spacing \cite{CI}.

The central theme in the above examples is that the more information we know about the zeros of the $L$-functions, the more we can say about arithmetically important questions. We started with just knowledge of the zeros on the line ${\rm Re}(s) = 1$, and then extended to GRH and all non-trivial zeros having real part 1/2, and then went beyond that to the distribution of the zeros on that critical line. In this chapter we expand on this last theme, and explore the distribution of zeros of $L$-functions on the critical line.

\subsection{$n$-Level Correlations and Random Matrix Theory}

While zeros of $L$-functions is a rich subject with an extensive history, our story on the number theory side begins in the 1970s with Montgomery's \cite{Mon} work on the pair correlation of the zeros of $\zeta(s)$. Given an increasing sequence of numbers $\{\alpha_j\}$ and $B \subset \R^{n-1}$ a compact box, the $n$-level correlation $R_n(B)$ is defined by
\begin{eqnarray}\label{eqnlevelcorr}
R_n(B) \ := \ \lim_{N \rightarrow \infty} \frac{  \# \left\{ (\alpha_{j_1}-\alpha_{j_2},\ \dots,\ \alpha_{j_{n-1}} -
\alpha_{j_n}) \in B,\ \ j_i \le N \right\}}{N},
\end{eqnarray} where the indices above are distinct. Instead of using a box (which is equivalent to a sharp cut-off) it's often technically easier to consider a similar version with a smooth test function (see \cite{RS}).

While knowing \emph{all} the $n$-level correlations allows one to determine all the neighbor spacings (see for example \cite{Meh}), computing these for arbitrary $B$ (or for any admissible test function) is well beyond current technology. There are, however, many important partial results. The first is the referred to one of Montgomery \cite{Mon}, who showed that for suitable test functions the $2$-level density agrees with the $2$-level density of eigenvalues of the Gaussian Unitary Ensemble (GUE). There are many ways to view these matrices. The easiest is that these are Hermitian matrices whose upper triangular entries are independently drawn from Gaussians (as the diagonal must be real, we draw from a different Gaussian for these entries than we do for the non-diagonal ones). An alternative definition, which explains the use of the word unitary, deals with the equality of the probability of choosing a matrix and its conjugation by a unitary matrix; note this is equivalent to saying the probability of a matrix is independent of the base used to write it down. From a physical point of view these matrices represent the Hamiltonian of a system. What matters are their eigenvalues, which correspond to the energy levels. While the entries of the matrix change depending on the basis used to write it down, the eigenvalues do not, which leads us to the unitary invariance condition. These and other matrix families had been extensively studied by Dyson, Mehta and Wigner among many others; see \cite{Con,FM,For,Ha,Meh,MT-B} and the multitude of references therein for more on the history and development of the subject.

This suggested a powerful connection between number theory and random matrix theory, which was further supported by Odlyzko's investigations \cite{Od1,Od2} showing agreement between the spacings of zeros of $\zeta(s)$ and the eigenvalues of the GUE. Subsequent work by Hejhal \cite{Hej} on the triple correlation of $\zeta(s)$ and Rudnick-Sarnak \cite{RS} on the $n$-level correlations for all automorphic cuspidal $L$-functions, again under suitable restrictions, provided additional support for the conjectured agreement between the limiting behavior of zeros of $L$-functions (as we move up the critical line) and eigenvalues of $N\times N$ matrices (as $N\to\infty$).

These results indicated a remarkable universality in behavior; while there are many random matrix ensembles, it appeared that only one was needed for number theory. A little thought, however, shows that this might not be the full story. The reason is that the $n$-level correlations are insensitive to the behavior of finitely many zeros. In other words, we can remove finitely many zeros and not change $R_n(B)$. This is particularly troublesome, as there are many problems where only a few zeros matter. For example, the Birch and Swinnerton-Dyer conjecture \cite{BS-D1,BS-D2} states that the order of vanishing of the $L$-function associated to the elliptic curve equals the rank of its Mordell-Weil group; thus in studies on this problem we \emph{only} care about what is happening at the central point, and not at all about what is happening far away on the critical line.

Later studies by Katz and Sarnak \cite{KaSa1,KaSa2} confirmed that more care is needed. The $n$-level correlations of the zeros of $L$-functions agree not only with those from the GUE, but also with those coming from the classical compact groups. The advantage of the latter is that the probability a matrix is chosen is derived from the Haar measure on the group, which is a canonical choice. This is in sharp contrast to the definition of the GUE, where we fix a probability distribution and choose independent entries from it, which begs the question why one distribution was chosen over another (for the GUE, the answer is that the Gaussian is forced upon us by our assumption of the probability being invariant under unitary transformations of the basis). They proved that as $N\to\infty$ the $n$-level correlations of the eigenvalues are the same for all the classical compact groups (unitary, symplectic, and orthogonal, split or not split by sign). Thus one could just as easily say that the zeros of $\zeta(s)$ behave like the eigenvalues of orthogonal matrices instead of the GUE.

This led Katz and Sarnak to introduce a new statistic that is both able to distinguish the different classical compact groups and which depends on the behavior of eigenvalues near 1. We briefly describe the comparisons between number theory and random matrix theory. If we assume the Riemann hypothesis then the non-trivial zeros have real part 1/2 and we may write them as $\rho_j = 1/2 + i \gamma_j$ for $\gamma_j$ real. On the random matrix theory side, the classical compact groups are unitary matrices, and we can therefore write their eigenvalues as $e^{i\theta_k}$ with $\theta_k$ real. From intuition gleaned from earlier results, as well as function field analogues, Katz and Sarnak were led to conjecture that in the appropriate limits the behavior of zeros near 1/2 agree with the behavior of eigenvalues near 1 (more generally, one can also compare values of $L$-functions and characteristic polynomials of matrices).

\subsection{$n$-level Densities and the Katz-Sarnak Philosophy}\label{sec:nleveldensitiesKatzSarnak}

Unfortunately, it is not possible to compare just the zeros of one $L$-function near the central point to the eigenvalues of one matrix. As in many problems in analytic number theory, we need to be able to execute some type of averaging and take some kind of limit in order to isolate out a main term and make progress. For the $n$-level correlations (or, equivalently, for Odlyzko's work on spacings between adjacent zeros), one $L$-function provides infinitely many zeros, and the average spacing between zeros at height $T$ is on the order of $1/\log T$. Thus if we go high up, we are essentially averaging over the zeros of that $L$-function, and can isolate out a universal, main term behavior. If instead we concentrate on the low-lying zeros, those near the central point, the situation is very different. To each $L$-function $L(s,f)$ we can associate a quantity, called the analytic conductor $c_f$, such that the first few zeros are of size $1/\log c_f$. If we rescale so that these zeros are of mean spacing one, then given any constant $C$ there are essentially a finite number (depending basically just on $C$) that are at most $C$.

In order to make progress we need to collect a large number of $L$-functions which should behave similar and are naturally connected. We call such a collection a \emph{family} of $L$-functions. The definition of what is a family is still a work in progress (see \cite{DM2} among others), but most natural collections of $L$-functions are. Examples include families of Dirichlet characters (either all of a given conductor, all whose conductor is in a given range say $[N, 2N]$, or just quadratic characters whose conductor is in a range), cuspidal newforms (and very important subsets, one or two parameter families of elliptic curves), symmetric powers of cusp forms, and so on. Collections that are not families would include arbitrary subsets, for example, cusp forms whose third Fourier coefficient is 2 modulo 5, or cusp forms whose first zero above the central point is at least twice the average. Typically as the conductors (or range) grows we have more and more $L$-functions in the family. The Katz-Sarnak philosophy is that if we take averages of statistics of zeros over the family then in the limit it will converge and agree with the corresponding statistic for the scaling limit of a classical compact group as the matrix size tends to infinity.

The main statistic we study in this paper is their $n$-level density. For convenience of exposition we assume the Generalized Riemann Hypothesis for $L(s,f)$ (and thus all the zeros are of the form $1/2 + i\gamma_{j;f}$ with $\gamma_{j;f}$ real), though the statistic below makes sense even if GRH fails. Let $\phi(x) = \prod_{j=1}^n \phi_j(x_j)$ where each $\phi_j$ is an even Schwartz functions such that the Fourier transforms \be \widehat{\phi_j}(y) \ := \ \int_{-\infty}^\infty \phi_j(x) e^{-2\pi i xy} dx\ee are compactly supported. The $n$-level density for $f$ with test function $\phi$ is
\begin{eqnarray}
D_n(f,\phi) & \ = \ &  \sum_{j_1, \dots, j_n \atop j_\ell \neq j_m} \phi_1\left(L_f \gamma_{j_1;f}\right) \cdots \phi_n\left(L_f
\gamma_{j_n;f}\right),
\end{eqnarray} where $L_f$ is a scaling parameter which is frequently related to the conductor. The idea is to average over similar $f$, and use the explicit formula to relate this sum over zeros to a sum over the Fourier coefficients of the $L$-functions. See for example \cite{ILS}, the seminal paper in the subject and the first to explore these questions, and see \cite{RS} for a nice derivation of the explicit formula for general automorphic forms.

The subject is significantly harder if the conductors vary in the family, as then we cannot just pass the averaging over the forms through the test function to the Fourier coefficients of the $L$-functions. If we only care about the 1-level density, then we may rescale the zeros by using the average log-conductor instead of the log-conductor; as this is the primary object of study below we shall rescale all the $L$-functions in a family by the same quantity, which we denote $R$, and we emphasize this fact by writing $D_n(f,\phi,R)$. For more on these technical issues, see \cite{Mil0, Mil1}, which studies families of elliptic curves where the variation in conductors must be treated. There it is shown that if the conductors vary within a family then this global renormalization leads to problems, and in the $2$-level computations terms emerge where we cannot just pass the averaging through the test function. In some problems, however, it is important to compute the $n$-level densities. One application is to obtain significantly better bounds on the order of vanishing at the central point (see \cite{HM}). Another is to distinguish orthogonal candidates if the $1$-level can only be computed for small support; we will elaborate on this below.

Given a family $\mathcal{F} = \cup_N \mathcal{F}_N$ of $L$-functions with conductors tending to infinity, the $n$-level density $D_n(\mathcal{F},\phi,R;w)$ with test function $\phi$, scaling $R$ and a non-negative weight function $w$ is defined by \be D_n(\mathcal{F},\phi,R;w) \ := \ \lim_{N \to \infty} \frac{\sum_{f\in \mathcal{F}_N} w(f) D_n(f,\phi,R;w)  }{ \sum_{f\in \mathcal{F}_N} w(f)}. \ee The advantage of this statistic is that individual zeros can contribute in the limit, with most of the contribution coming from the zeros near the central point due to the rapid decay of the test functions. Further, as we are averaging over similar forms there is a hope that there is a nice limiting behavior. In most applications we really have $w_T(t) = w(t/T)$, but suppress the subscript $T$ as its understood.

Katz and Sarnak \cite{KaSa1, KaSa2} proved that the $n$-level density is different for each classical compact group, and found nice determinant expansions for them. Set $K(y) := \kkot{y}$ and $K_\epsilon(x,y) := K(x-y) + \epsilon K(x+y)$ for $\epsilon = 0, \pm 1$. They proved that if $G_N$ is either the family of $N\times N$ unitary, symplectic or orthogonal families (split or not split by sign), the $n$-level density for the eigenvalues converges as $N\to\infty$ to \bea & & \int \cdots \int \phi(x_1,\dots,x_n) W_{n,G}(x_1,\dots,x_n) dx_1 \cdots dx_n \nonumber\\ & & \ \ \ \ \  \ = \ \int \cdots \int \hphi(y_1,\dots,y_n) \widehat{W}_{n,G}(y_1,\dots,y_n) dy_1 \cdots dy_n, \eea where
\begin{eqnarray}\label{eqdensitykernels}
W_{m,{\rm SO(even)}}(x) &\ =\ & \det (K_1(x_i,x_j))_{i,j\leq m} \nonumber\\
W_{m,{\rm SO(odd)}}(x) & = & \det (K_{-1}(x_i,x_j))_{i,j\leq
m}  + \sum_{k=1}^m
\delta(x_k) \det(K_{-1}(x_i,x_j))_{i,j\neq k} \nonumber\\
W_{m, {\rm SO}}(x) & = & \foh W_{m,{\rm SO(even)}}(x) + \foh W_{m,{\rm SO(odd)}}(x)
\nonumber\\ W_{m, {\rm U}}(x) & = & \det (K_0(x_i,x_j))_{i,j\leq
m} \nonumber\\ W_{m, {\rm Sp}}(x) &=& \det(K_{-1}(x_i,x_j))_{i,j\leq m}.
\end{eqnarray} While these densities are all different, for the 1-level density with test functions whose Fourier transforms are supported in $(-1, 1)$ the three orthogonal flavors cannot be distinguished from each other, though they can be distinguished from the unitary and symplectic. Explicitly, the Fourier Transforms for the $1$-level densities are
\begin{eqnarray}
\widehat{W_{1,{\rm SO(even)}} }(u) & = & \delta_0(u) + \foh \eta(u)
\nonumber\\ \widehat{W_{1,SO} }(u) & = & \delta_0(u) + \foh
\nonumber\\ \widehat{W_{1,{\rm SO(odd)}} }(u) & = & \delta_0(u) - \foh
\eta(u) + 1 \nonumber\\ \widehat{W_{1,Sp} }(u) & = & \delta_0(u) -
\foh \eta(u) \nonumber\\ \widehat{W_{1,U} }(u) & = & \delta_0(u),
\end{eqnarray} where $\eta(u)$ is $1$, $1/2$, and $0$ for $|u|$ less than $1$, $1$, and greater than $1$, and $\delta_0$ is the standard Dirac
Delta functional. Note that the first three densities agree for $|u| < 1$ and split (ie, become distinguishable) for $|y| \geq 1$. Thus in order to uniquely specify a symmetry type among the three orthogonal candidates, one either needs to obtain results for support exceeding $(-1, 1)$, or compute the $2$-level density, as that is different for the three orthogonal groups for arbitrarily small support \cite{Mil0, Mil1}.

The Katz-Sarnak Density Conjecture states that the behavior of zeros near the central point in a family of $L$-functions (as the conductors tend to infinity) agrees with the behavior of eigenvalues near 1 of a classical compact group (as the matrix size tends to infinity). For suitable test functions, this has been verified in many families, including Dirichlet characters, elliptic curves, cuspidal newforms, symmetric powers of ${\rm GL}(2)$ $L$-functions, and certain families of ${\rm GL}(4)$ and ${\rm GL}(6)$ $L$-functions; see for example \cite{DM1, DM2, ER-GR, FiM, FI, Gao, Gu, HM, HR, ILS, KaSa2, LM, Mil1, MilPe, OS, RR, Ro, Rub, Ya, Yo2}. This correspondence between zeros and eigenvalues allows us, at least conjecturally, to assign a definite symmetry type to each family of $L$-functions (see \cite{DM2, ShTe} for more on identifying the symmetry type of a family).

For this work, the most important families studied to date are holomorphic cusp forms. Using the Petersson formula (and a delicate analysis of the exponential sums arising from the Bessel-Kloosterman term), Iwaniec, Luo, and Sarnak \cite{ILS} proved that the limiting behavior of the zeros near the central point of holomorphic cusp forms agrees with that of the eigenvalues of orthogonal matrices for suitably restricted test functions. In this chapter we look at the other ${\rm GL}_2/\Q$ family of $L$-functions, Maass waveforms.



\section{Statement of Main Results}\label{sec:mainresultsAAILMZ}

We first describe the needed normalizations and notation for our families of Maass forms, and then conclude by stating our new results and sketching the arguments. The beginning of the proofs are similar to that in all families studied to date: one uses the explicit formula to convert sums over zeros to sums over the Fourier coefficients of the $L$-functions. The difficulty is averaging over the family. In order to obtain support beyond $(-1, 1)$, we have to handle some very delicate exponential sums; these arise from the Bessel-Kloosterman term in the Kuznetsov trace formula. To facilitate applying it, we spend a lot of time choosing tractable weights. This is similar to previous work on cuspidal newforms where the harmonic weights were used to simplify the application of the Petersson trace formula. It is possible to remove these weights, and this is done in \cite{ILS}. For some applications it is important to have unweighted families, in order to talk about the percentage of forms that vanish at the central point to a given order (see \cite{HM}); for our purposes we are primarily interested in obtaining large enough support to uniquely determine the symmetry type, and thus choose our weight functions accordingly.

\subsection{Normalizations and Notation}

We quickly recall the basic properties of Maass forms (see \cite{Iw2,IK,Liu,LiuYe} for details), and then review the 1-level density from the last section with an emphasis on the important aspects for the subsequent computations. We use the standard conventions. Specifically, by $A\ll B$ we mean $|A|\leq c|B|$ for a positive constant $c$. Similarly, $A\asymp B$ means $A\ll B$ and $A\gg B$. We set $e(x):=\exp(2\pi i x)$, and define the Fourier transform of $f$ by \be\widehat{f}(\xi)\ :=\ \int_\R f(x) e(-x\xi) dx.\ee


Let $u$ be a Maass cusp form on $\Gamma_0(N)$, $N$ square-free, with Laplace eigenvalue $\lambda_u =: (\frac{1}{2} + it_u)(\frac{1}{2} - it_u)$. Selberg's $3/16$ths theorem implies that we may take $t_u\geq 0$ or $t_u\in [0,\frac{1}{4}]i$. Next we Fourier expand $u$ as follows: \begin{align}u(z)\ =\ y^{1/2}\sum_{n\neq 0} a_n(u)K_{s-1/2}(2\pi |n| y)e(ny).\end{align} Let \begin{align}\lambda_n(u)\ :=\ \frac{a_n(u)}{\cosh(\pi t_u)^{1/2}}.\end{align} We normalize $u$ so that $\lambda_1(u) = 1$.

The $L$-function associated to $u$ is \begin{equation}L(s,u) \ :=\ \sum_{n\geq 1} \lambda_n n^{-s}.\end{equation} By results from Rankin-Selberg theory the $L$-function is absolutely convergent in the right half-plane ${\Re}(s) > 1$ (one could also use the work of Kim and Sarnak \cite{K,KSa} to obtain absolutely convergent in the right half-plane ${\Re}(s) > 71/64$, which suffices for our purposes). These $L$-functions analytically continue to entire functions of the complex plane, satisfying the functional equation \begin{equation}\Lambda(s,u)\ =\ (-1)^\eps\Lambda(1-s,u),\end{equation} with \begin{equation}\Lambda(s,u) \ :=\ \pi^{-s}\Gamma\left(\frac{s + \eps + it}{2}\right)\Gamma\left(\frac{s+\eps-it}{2}\right)L(s,u).\end{equation} Factoring \begin{equation}1 - \lambda_pX + X^2\ =:\ (1-\alpha_pX)(1-\beta_pX)\end{equation} at each prime (the $\alpha_p,\beta_p$ are the Satake parameters at $p$), we get an Euler product \begin{equation}L(s,u)\ =\ \prod_p (1-\alpha_pp^{-s})^{-1}(1-\beta_pp^{-s})^{-1},\end{equation} which again converges for ${\Re}(s)$ sufficiently large.

For the remainder of the paper $\mathcal{B}_N$ denotes an orthogonal basis of Maass cusp forms on $\Gamma_0(N)$, all normalized so that $\lambda_1 = 1$; thus $\mathcal{B}_N$ is \emph{not} orthonormal under the Petersson inner product on the space. Note we do not take a basis of \emph{newforms} --- that is, the delicate sieving out of oldforms as in \cite{ILS} is not done. Of course such sieving is easy for $N$ prime by the relevant Weyl law.

We use the notation $\Avg(A;w)$ to denote the average of $A$ over $\mathcal{B}_N$ with each element $u\in \mathcal{B}_N$ given weight $w(u)$. That is, \be\Avg(A;w)\ :=\ \frac{\sum_{u\in \mathcal{B}_N} A(u)w(u)}{\sum_{u\in \mathcal{B}_N} w(u)}.\ee

Our main statistic for studying the low-lying zeros (i.e., the zeros near the central point) is the 1-level density; we quickly summarize the needed definitions and facts from \S\ref{sec:nleveldensitiesKatzSarnak}. Let $\phi$ be an even Schwartz function such that the Fourier transform $\widehat{\phi}$ of $\phi$ has compact support; that is,
\be
\widehat{\phi}(y)\ =\ \int_{-\infty}^\infty \phi(x)e^{-2\pi i xy}dx
\ee
and there is an $\eta<\infty$ such that $\widehat{\phi}(y) = 0$ for $y$ outside $(-\eta, \eta)$.

The $1$-level density of the zeros of $L(s, u)$ is
\begin{equation} \label{eqn: 1leveldef}
D_1(u, \phi, R) \ = \
\sum_{\rho}
 \phi \left(\frac{\log R}{2\pi}\gamma \right),
\end{equation}
where $\rho = 1/2+i\gamma$ are the nontrivial zeros of $L(s, u)$, and $\log R$ is a rescaling parameter related to the average log-conductor in the weighted family, whose choice is forced upon us by \eqref{eq:RforcedtobeTTN}. Under GRH all $\gamma$ are real and the zeros can be ordered; while GRH gives a nice interpretation to the 1-level density, it is not needed for our purposes. As $\phi$ is a Schwartz function, most of the contribution comes from the zeros near the central point $s = 1/2$. The different classical compact groups (unitary, symplectic, and orthogonal) have distinguishable 1-level densities for arbitrarily small support; however, the 1-level densities for the even and odd orthogonal matrix ensembles are equal for test functions whose Fourier transforms are supported in $(-1, 1)$. There are two solutions to this issue. One possibility is to perform a more detailed analysis and ``extend support''. The other is to study the 2-level density, which Miller \cite{Mil0, Mil1} showed distinguishes the orthogonal ensembles for arbitrarily large support. For some of the families studied below we are able to calculate the support beyond $(-1,1)$, and we may thus determine which of the orthogonal groups should be the symmetry group; for the other families our support is too limited and we instead study the 2-level density.

\subsection{Main Results} Similar to how the harmonic weights facilitate applications of the Petersson formula to average the Fourier coefficients of cuspidal newforms (see for instance \cite{ILS,MilMo}), we introduce nice, even weight functions to smooth the sum over the Maass forms. As we will see below, some type of weighting is necessary in order to restrict to conductors of comparable size. While our choice does not include the characteristic function of $[T, 2T]$, we are able to localize for the most part to conductors near $T$, and are able to exploit smoothness properties of the weight function in applications of the Kuznetsov trace formula. Further, in problems such as these the primary goal is to have as large support as possible for the Fourier transform of the test function that hits the zeros. For more on these issues see \cite{AM}, where Alpoge and Miller impose even more restrictions on the weight functions, which allow them to increase the support.

We consider the averaged one-level density weighted by two different weight functions of ``nice'' analytic properties. Let $\widehat{H}\in C^\infty\left(\left(-\frac{1}{4}, \frac{1}{4}\right)\right)$ be an even smooth bump function of compact support on the real line, and let $H$ be its Fourier transform. We may of course (by applying this construction to a square root --- recall that the support of a convolution is easily controlled) take $H\geq 0$. We may also take $H$ to have an order $K$ zero at $0$. Let \begin{align}\label{eq:HTdefn} H_T(r)\ :=\ H\left(\frac{r}{T}\right).\end{align} This is essentially supported in a band of length $\asymp T$ about $\pm T$.


Next, in the same way, let $\widehat{h}\in C^\infty\left(\left(-\frac{1}{4},\frac{1}{4}\right)\right)$ be even. We also require $h$ to have an order at least $8$ zero at $0$. Note that, by the same process as above, we may take $h(x)\geq 0$ for all $x\in \R$ and also (by Schwarz reflection) $h(ix)\geq 0$ for all $x\in \R$. Let $T$ be a positive odd integer. We let \begin{align}\label{eq:hTdefn} h_T(r)\ :=\ \frac{\frac{r}{T} h\left(\frac{ir}{T}\right)}{\sinh\left(\frac{\pi r}{T}\right)}.\end{align} This is the same test function used in \cite{AM}, and is essentially supported in a band of length $\asymp T$ about $\pm T$.

By trivially bounding the Fourier integral we observe that \begin{align}H(x+iy),\ h(x+iy)\ \ll \  \exp\left(\frac{\pi |y|}{2}\right).\end{align} Hence \begin{align}H_T(ir)\ \ll \  \exp\left(\frac{\pi |r|}{2T}\right),\end{align} and, using $\sinh(x)\gg e^{|x|}$, we find \begin{align}h_T(r)\ \ll \  \exp\left(-\frac{\pi |r|}{4T}\right).\end{align} These will both be useful in what follows.

In one-level calculations we will take an even Schwartz function $\phi$ such that $\widehat{\phi}$ is supported inside $[-\eta,\eta]$. We suppress the dependence of constants on $h,H,\phi$ and $\eta$ (as these are all fixed), but not $T$ or the level $N$ since one or both of these will be tending to infinity.

We weight each element $u\in \mathcal{B}_N$ by either $H_T(t_u)/||u||^2$ or $h_T(t_u)/||u||^2$, where \begin{align}||u||^2\ = \ ||u||_{\Gamma_0(N)\backslash\hfrak}^2\ = \ \frac{1}{[\SL_2(\Z):\Gamma_0(N)]}\int_{\Gamma_0(N)\backslash\hfrak} u(z)\frac{dxdy}{y^2}\end{align} is the $L^2$-norm of $u$ on the modular curve $Y_0(N)$, and, as before, $\lambda_u = \frac{1}{4} + t_u^2$ is the Laplace eigenvalue of $u$. Recall that \begin{align}[\SL_2(\Z):\Gamma_0(N)]\ =: \ \nu(N)\ = \ \prod_{p\vert N} (p + 1).\end{align}

The averaged weighted one-level density may thus be written (we will see in \eqref{eq:RforcedtobeTTN} that $R\asymp T^2N$ is forced) \begin{align}{D}_1(\mathcal{B}_N,\phi,R;w)\ :=\ \Avg\left(D_1(u,\phi,R); w(t_u)/||u||^2\right),\end{align} where $w(t_u)$ is either $H_T(t_u)$ or $h_T(t_u)$.

The main question is to determine the behavior of ${D}_1(\mathcal{B}_N,\phi,R;w)$ as either the level $N$ or the weight parameter $T$ tends to infinity; specifically, one is generally interested in the corresponding symmetry group. There are now several works \cite{DM2, KaSa1, KaSa2, ShTe} which suggest ways to determine the symmetry group. For our family, they suggest the following conjecture.

\begin{conj} Let $h_T$ be as in \eqref{eq:hTdefn}, $\phi$ an even Schwartz function with $\widehat{\phi}$ of compact support, and $R \asymp T^2 N$. Then \begin{align}\lim_{R\to\infty}{D}_1\left(\mathcal{B}_N,\phi,R;w\right)\ = \ \int_\R \phi(t)W_{1, {\rm SO}}(t)dt,\end{align} where $W_{1, {\rm SO}}:=1 + \frac{1}{2}\delta_0$. In other words, the symmetry group associated to the family of Maass cusp forms of level $N$ is orthogonal.
\end{conj}

In \cite{AM} the above conjecture is shown for $N = 1, w = h_T$, with the extra restrictions that $h$ has $2K \ge 8$ zeros at the origin and $\mathrm{supp}(\widehat{\phi})\subseteq (-2 + \frac{2}{2K+1}, 2 - \frac{2}{2K+1})$. Our first result here is in the case where the level $N$ tends to infinity (remember that $N$ must be square-free), $T$ and $K$ are fixed, and $w = w_T$ equals either $h_T$ or $H_T$.

\begin{thm}\label{maintheoremone} Fix $T$ and $K$ and let $R \asymp T^2 N$. Let $H$ be an even, non-negative function with $K$ zeros at 0 and Fourier transform $\widehat{H}\in C^\infty\left(\left(-\frac{1}{4},\frac{1}{4}\right)\right)$, and let $h$ be an even function with 8 zeros at 0 and $\widehat{h}\in C^\infty\left(\left(-\frac{1}{4},\frac{1}{4}\right)\right)$. Let the weights $w=w_T$ be either $H_T$ or $h_T$, where these are the functions given by \eqref{eq:HTdefn} and \eqref{eq:hTdefn}, respectively. Let $\phi$ be an even Schwartz function with $\mathrm{supp}(\widehat{\phi})\subseteq (-\frac{3}{2},\frac{3}{2})$. Then \begin{align}\lim_{N\to\infty\atop N\ {\rm square-free}} {D}_1(\mathcal{B}_N,\phi,R;w_T)\ = \ \int_\R \phi(t)W_{1, {\rm SO}}(t)dt.\end{align}
\end{thm}

Notice that the support in Theorem \ref{maintheoremone} exceeds $(-1,1)$, and thus we have uniquely specified which orthogonal group is the symmetry group of the family.

Next we investigate the case where $N$ is fixed and $T$ tends to infinity through odd values. For ease of exposition we take $N=1$.

\begin{thm}\label{maintheoremtwo} Let $h$ be an even function with 8 zeros at 0 and $\widehat{h}\in C^\infty\left(\left(-\frac{1}{4}, \frac{1}{4}\right)\right)$, and define $h_T$ as in \eqref{eq:hTdefn}. Let $\phi$ be an even Schwartz function with $\mathrm{supp}(\widehat{\phi})\subseteq (-1,1)$, and take $R \asymp T^2 N$ with $N=1$. Then \begin{align}\lim_{T\to\infty\atop {T\ {\rm odd}}} {D}_1(\mathcal{B}_1,\phi,R;h_T)\ = \ \int_\R \phi(t)W_{1, {\rm SO}}(t)dt.\end{align}
\end{thm}

We also get a similar, though slightly weaker, result for the weight function $w=H_T$ if we allow $K = \ord_{z=0} H(z)$ to vary. For the argument given we invoke the work of \cite{ILS} twice, since we reduce the Bessel-Kloosterman term of the Kuznetsov trace formula to sum of Kloosterman terms arising in the Petersson trace formula. Since \cite{ILS} use GRH (specifically, for Dirichlet $L$-functions and $L$-functions associated to symmetric squares of holomorphic cusp forms), we must, too.

\begin{thm}\label{maintheoremthree} Assume GRH for Dirichlet $L$-functions and symmetric squares of holomorphic cusp forms of level 1. Let $H$ be an even, non-negative function with $K$ zeros at 0 and Fourier transform $\widehat{H}\in C^\infty\left(\left(-\frac{1}{4},\frac{1}{4}\right)\right)$, and let the weights $w=w_T$ be $H_T$, which is given by \eqref{eq:HTdefn}. Let $\phi$ be an even Schwartz function with $\Supp(\widehat{\phi})\subseteq (-1 + \frac{1}{5+2K}, 1 - \frac{1}{5+2K})$. Take $R \asymp T^2 N$. Then \begin{align}\lim_{T\to\infty} {D}_1(\mathcal{B}_1,\phi,R;H_T)\ = \ \int_\R \phi(t)W_{1, {\rm SO}}(t)dt.\end{align}
\end{thm}

Notice the support in Theorems \ref{maintheoremtwo} and \ref{maintheoremthree} is too small to uniquely determine which orthogonal symmetry is present (this is because the one-level densities of the orthogonal flavors all agree inside $(-1,1)$). At the cost of more technical arguments, Alpoge and Miller \cite{AM} are able to extend the support beyond $(-1,1)$ when weighting by $h_T$, thereby determining the symmetry group to be orthogonal. In this work we instead compute the 2-level density, which provides a second proof that the symmetry type of the family of Maass cusp forms on $\SL_2(\Z)$ is orthogonal. The 2-level density is defined in \eqref{two level definition}. As any support for the 2-level density suffices to uniquely determine the symmetry group, we do not worry about obtaining optimal results.

\begin{thm}\label{maintheoremfour} Let $w_T$ equal $h_T$ or $H_T$, $R \asymp T^2$, and let \be \mathcal{N}(-1)\ :=\ \frac{1}{\sum_{u\in \mathcal{B}_1} \frac{w_T(t_u)}{||u_u||^2}}  \sum_{u: (-1)^{\eps+\eps'+1} = -1} \frac{w_T(t_u)}{||u||^2} \ee be the weighted percentage of Maass forms in $\mathcal{B}_1$ with odd functional equation. Write \begin{align}D_2(\mathcal{B}_1,\phi_1,\phi_2,R;w_T)\ :=\ \Avg\left(D_2(\mathcal{B}_1,u,\phi_1,\phi_2,R); w_T(t_u) / ||u||^2\right), \end{align} with \begin{eqnarray}\label{two level definition} D_2(\mathcal{B}_1,u,\phi_1,\phi_2,R;w_T)&\ :=\ & \sum_{i\neq \pm j} \phi_1\left(\frac{\log{R}}{2\pi}\gamma_i\right)\phi_2\left(\frac{\log{R}}{2\pi}\gamma_j\right) \nonumber\\&=& D_1(\mathcal{B}_1,u,\phi_1,R;w_T) D_1(\mathcal{B}_1,u,\phi_2,R;w_T) \nonumber\\ & & \ \ -\ 2D_1(\mathcal{B}_1,u,\phi_1\phi_2,R;w_T)\nonumber\\ & & \ \ +\ \delta_{(-1)^{\eps+\eps'+1},-1} \phi_1(0)\phi_2(0),\end{eqnarray} the average 2-level density of the weighted family of level 1 Maass cusp forms, and the 2-level density of $u\in \mathcal{B}_1$, respectively. Let $f \ast g$ denote the convolution of $f$ and $g$. Then, for $\eps\ll 1$ and $\widehat{\phi_1},\widehat{\phi_2}$ even Schwartz functions supported in $(-\eps,\eps)$, \begin{align}\lim_{T\to\infty} D_2(\phi_1,\phi_2,R;w_T) &= \left(\frac{\phi_1(0)}{2} + \widehat{\phi_1}(0)\right)\left(\frac{\phi_2(0)}{2} + \widehat{\phi_2}(0)\right) \nonumber\\&\quad\quad+ 2\int_\R |x|\widehat{\phi_1}(x)\widehat{\phi_2}(x) dx \nonumber\\&\quad\quad- (1-\mathcal{N}(-1))\phi_1(0)\phi_2(0) - 2(\widehat{\phi_1} \ast \widehat{\phi_2})(0),\end{align} agreeing with the 2-level density of the scaling limit of an orthogonal ensemble with proportions of $\mathcal{N}(-1)$ ${\rm SO}({\rm odd})$ matrices and $1-\mathcal{N}(-1)$ ${\rm SO}({\rm even})$ matrices.
\end{thm}
A similar result holds for $\mathcal{B}_N$ --- all the calculations will be standard given our work on the one-level densities.


\subsection{Outline of Arguments}

By a routine application of the explicit formula we immediately reduce the problem to studying averages of Hecke eigenvalues over the space of Maass cusp forms of level $N$. For this we apply the Kuznetsov trace formula, as found in \cite{KL}. We are quickly reduced to studying a term of shape \begin{align}\nu(N)\sum_{c\geq 1} \frac{S(m,1;cN)}{cN}\int_\R J_{2ir}\left(\frac{4\pi\sqrt{m}}{cN}\right)\frac{r w_T(r)}{\cosh(\pi r)}dr.\end{align} In all cases the idea is to move the contour from $\R$ to $\R - iY$ with $Y\to\infty$. The properties of the weights $h_T$ or $H_T$ ensure that the integral along the moving line vanishes in the limit, so all that is left in place of the integral is the sum over poles, of shape (up to negligible error in the case of $h_T$, which also has poles of its own) \begin{align}\sum_{k\geq 0} (-1)^k J_{2k+1}\left(\frac{4\pi\sqrt{m}}{cN}\right) (2k+1)w_T\left(\frac{2k+1}{2}i\right).\end{align} Now the $N\to\infty$ limit is very easy to take, as all the Bessel functions involved have zeros at $0$. So the term does not contribute (the total mass is of order $T^2\nu(N)$, canceling the $\nu(N)$ out in front). With some care we arrive at Theorem \ref{maintheoremone}.

If instead we take $N=1$ and $w_T = H_T$, then, by standard bounds on Bessel functions, $J_{2k+1}(\frac{4\pi\sqrt{m}}{c})$ is very small for $k$ larger than $\asymp\frac{\sqrt{m}}{c}$. For us $\sqrt{m}$ will always be bounded in size by something that is $\asymp T^\eta$. Thus for $k$ smaller than this range, the Bessel term is still controlled but not too small. It is the term \begin{align}w_T\left(\frac{2k+1}{2}i\right)\ \asymp \  H\left(\frac{2k+1}{2T}i\right)\ \ll \  \left(\frac{k}{T}\right)^K\end{align} that is small. Upon taking \begin{align}K\ \gg \  \frac{1}{1-\eta}\end{align} it is in fact small enough to bound trivially. For slightly larger support we instead appeal to the bounds of \cite{ILS}
on sums of Kloosterman sums, which are derived from assuming GRH for Dirichlet $L$-functions. This gives Theorem \ref{maintheoremthree}. In fact, the expression we get is exactly a weighted sum of terms appearing in \cite{ILS} from the Kloosterman terms of Petersson formulas. It would be interesting to find a conceptual explanation for this.

The proof of Theorem \ref{maintheoremtwo} is a simplified version of the argument given in \cite{AM}, except considerably shortened --- instead of delicate analysis of exponential sums, we just use Euler-Maclaurin summation. As one would expect our support is thus smaller than that in \cite{AM}, but the argument and main ideas are significantly easier to see.

The proof of Theorem \ref{maintheoremfour} follows from the previous results and another application of the Kuznetsov formula, this time to the inner product of $T_{p^\ell}$ with $T_{q^{\ell'}}$ with $p,q$ primes.

\section{Preliminaries for the Proofs}

In this section we compute and analyze some expansions and resulting expressions that are useful in the proofs of our main theorems. We start in \S\ref{sec:calcaveonelevel} by using the explicit formula to relate the sum over zeros to sums over the Hecke eigenvalues of the associated cusp forms. The weights and normalizations are chosen to facilitate applying the Kuznetsov trace formula to these sums, which we do. After trivially handling several of the resulting terms, in \S\ref{sec:handlingbessel} we analyze the Bessel function integral that arises. We then use these results in \S\ref{sec:proofsofmaintheorems} to prove the stated theorems.

\subsection{Calculating the averaged one-level density}\label{sec:calcaveonelevel}

We first quickly review the computation of the explicit formula; see \cite{ILS, RS} for details. Let $u\in \mathcal{B}_N$, and for an even Schwartz test function $\phi$ set \begin{align}\Phi(s)\ :=\ \phi\left(\frac{\left(s-\frac{1}{2}\right)\log{R}}{2\pi i}\right).\end{align} Consider \begin{align}\int_{\sigma = \frac{3}{2}} \Phi(s)\frac{\Lambda'}{\Lambda}(s,u) ds.\end{align} By moving the integration to $\sigma = -\frac{1}{2}$ and applying the functional equation, we find that \begin{align}2\int_{\sigma = \frac{3}{2}} \frac{\Lambda'}{\Lambda}(s,u)\Phi(s) ds\ = \ D_1(u,\phi)\end{align} (use the rapid decay of $\phi$ along horizontal lines and Phragmen-Lindel\"of to justify the shift). After expanding the logarithmic derivative out in the usual way, applying the Kim-Sarnak bound, and noticing that $\lambda_{p^2} = \lambda_p^2 - \chi_0(p)$ for $\chi_0$ the principal character modulo $N$, this equality simplifies to
\begin{align}
D_1(u,\phi) &\ = \  \frac{\phi(0)}{2} + \widehat{\phi}(0)\left(\frac{\log{N} + \log(1+t_u^2)}{\log{R}}\right) + O\left(\frac{\log\log{R} + \log\log{N}}{\log{R}}\right) \nonumber\\&\quad \ \ +\ 2\sum_{\ell=1}^2\sum_p \frac{\lambda_{p^\ell}\log{p}}{p^{\frac{\ell}{2}}\log{R}}\widehat{\phi}\left(\frac{\ell\log{p}}{\log{R}}\right).
\end{align}

Thus, if $w_T$ is essentially supported on $\asymp T$ (as are $h_T$ and $H_T$), the averaged one-level density is (since $||u||\asymp 1$ under our normalizations, by \cite{Smi})
\begin{align}\label{eq:RforcedtobeTTN}
{D}_1(\mathcal{B}_N,\phi,R;w_T) &\ = \  \frac{\phi(0)}{2} + \widehat{\phi}(0)\left(\frac{\log(T^2N)}{\log{R}}\right) + O\left(\frac{\log\log{R} + \log\log{N}}{\log{R}}\right) \nonumber\\&\quad+ 2\sum_{\ell=1}^2\sum_p \frac{\log{p}}{p^{\frac{\ell}{2}}\log{R}}\widehat{\phi}\left(\frac{\ell\log{p}}{\log{R}}\right)\Avg(\lambda_{p^\ell};w_T).
\end{align} Notice the above computation tells us the correct scaling to use is $R\asymp T^2N$.

The difficulty is in determining the averages over Hecke eigenvalues. For this we use the Kuznetsov formula for $\mathcal{B}_N$. Let $w_T$ equal $h_T$ or $H_T$.

\begin{thm}[Kuznetsov trace formula (see \cite{KL}, page 86)]\label{thm:kuznetsovtraceformula} Let $m\in \Z^+$. Then
\begin{align}\label{KTF}
\sum_{u\in \mathcal{B}_N} \frac{\lambda_m(u)}{||u||^2}w_T(t_u) &\ = \  \frac{\delta_{m,1} \nu(N)}{\pi^2}\int_\R rw_T(t)\tanh(\pi r)dr \nonumber\\&\quad- \frac{1}{\pi}\sum_{(i_p)_{p\vert N}\in \{0,1\}^{\omega(N)}}\int_\R \frac{\tilde{\sigma}_{ir}(m,(i_p))\overline{\tilde{\sigma}_{ir}(1,(i_p))} m^{ir} w_T(r)}{||(i_p)||^2 |\zeta(1 + 2ir)|^2} dr \nonumber\\&\quad+ \frac{2i}{\pi}\frac{\nu(N)}{N}\sum_{c\geq 1} \frac{S(m,1;Nc)}{c}\int_\R J_{2ir}\left(\frac{4\pi\sqrt{m}}{Nc}\right) \frac{rw_T(r)}{\cosh(\pi r)}dr,
\end{align}
where $S$ is the usual Kloosterman sum, $(i_p)_{p\vert N}$ runs through all $0\leq i_p\leq 1$ with $p$ ranging over the prime factors of $N$, \begin{eqnarray}& & \tilde{\sigma}_{ir}(a,(i_p)) \ := \ \nonumber\\ & & \left(\prod_{p\vert N} p^{i_p}\right)^{-1-2ir}\sum_{d\vert a} \frac{\chi_0(d\bmod{\prod_{p\vert N} p^{1-i_p}})}{d^{2ir}}\sum_{f\in \left(\Z/\prod_{p\vert N} p^{i_p}\Z\right)^\times} e\left(\frac{af}{d\prod_{p\vert N} p^{i_p}}\right),\ \ \ \ \ \end{eqnarray} $J$ is the usual Bessel function, and \begin{align}||(i_p)||^2\ = \ \prod_{p\vert N} \frac{p^{1-i_p}}{1+p}\ = \ \frac{N}{\nu(N) \prod_{p\vert N} p^{i_p}}.\end{align}
\end{thm}

In our applications we always have $(m,N) = 1$ since we only take $m = 1, p$, or $p^2$, which means that the contribution from the principal character in the definition of $\tilde{\sigma}_{ir}$ may be ignored. Also the inner sum in $\tilde{\sigma}_{ir}(a,(i_p))$ is of the form \begin{align}\sum_{\xi\in (\Z/n\Z)^\times} e\left(\frac{\xi}{n}\right)\ = \ \mu(n)\ \ll \  1.\end{align} Hence, bounding trivially and noting that our $m$ have at most three divisors, we find \begin{align}\tilde{\sigma}_{ir}(a,(i_p))\ \ll \  \left(\prod_{p\vert N} p^{i_p}\right)^{-1}.\end{align} Also, by work of de la Vallee Poussin on the prime number theorem, $\zeta(1+2ir)\gg \log(2+|r|)^{-1}$. Hence the second term in \eqref{KTF}, the Eisenstein contribution, is
\begin{align}
&\ \ll \  \frac{\nu(N)}{N}\sum_{(i_p)_{p\vert N}} \frac{1}{\prod_{p\vert N} p^{i_p}}\int_\R w_T(r)\log(2+|r|)dr\nonumber\\&\ \ll \  \frac{\nu(N) T\log{T}}{N}\prod_{p\vert N} 1 + \frac{1}{p} \nonumber\\&\ = \  \left(\frac{\nu(N)}{N}\right)^2 T\log{T}.
\end{align}

In our applications we will always divide these expressions by the corresponding expression with $m=1$, which gives the total mass of the family (``the denominator'' in the sequel). We will see that it is of order $\asymp T^2\nu(N)$ (see Corollary \ref{cor:denominatorsize}). Hence, since it will be divided by something of order $\asymp T^2\nu(N)$, the Eisenstein contribution is thus negligible for $N$ or $T$ large.

Note that the diagonal term (that is, the first term of \eqref{KTF}, with $m=1$) is \begin{align}\nu(N)\int_\R rw_T(r)\tanh(\pi r)dr\ \asymp \  T^2\nu(N).\end{align} Hence to show the claim about the total mass it suffices to bound the last term of \eqref{KTF} in the case of $m=1$.

We have therefore reduced the computation of the weighted 1-level density to understanding the ``Bessel-Kloosterman'' terms. We isolate this result below.

\begin{lem}\label{lem:kuznetsovreduction}
If $m = 1, p$ or $p^2$ is coprime to $N$ and $w_T$ equals $h_T$ or $H_T$, then
\begin{align}
\sum_{u\in \mathcal{B}_N} \frac{\lambda_m(u)}{||u||^2}w_T(t_u) &\ = \  \delta_{m,1}\cdot \left(\asymp T^2 N\right) \nonumber\\&\quad+ O\left(\left(\frac{\nu(N)}{N}\right)^2 T\log{T}\right) \nonumber\\&\quad+ \frac{2i}{\pi}\frac{\nu(N)}{N}\sum_{c\geq 1} \frac{S(m,1;Nc)}{c}\int_\R J_{2ir}\left(\frac{4\pi\sqrt{m}}{Nc}\right) \frac{rw_T(r)}{\cosh(\pi r)}dr,
\end{align}
where $\delta_{m,1}\cdot \left(\asymp T^2 N\right)$ is the product of a term on the order of $T^2N$ with Kronecker's delta.
\end{lem}

\subsection{Handling the Bessel integral}\label{sec:handlingbessel}

As in \cite{AM}, the technical heart of the analysis of the Kuznetsov formula is the following claim, which relies on the analytic properties of $h_T$ and $H_T$; see \S\ref{sec:appendixAcontourintegration} (Appendix I) for a proof.

\begin{prop}\label{prop:usefulexpansionbesselterm} Let $T$ be an odd integer. Let $X\leq T$. Let $w_T$ equal $h_T$ or $H_T$, where these are the weight functions from Theorems \ref{maintheoremone} through \ref{maintheoremfour}. Then \begin{align}
\int_\R J_{2ir}(X)\frac{rw_T(r)}{\cosh(\pi r)}dr &\ = \  c_1\sum_{k\geq 0} (-1)^k J_{2k+1}(X) (2k+1) w_T\left(\left(k+\frac{1}{2}\right)i\right) \nonumber\\&\quad\quad\quad\left[+ c_2 T^2 \sum_{k\geq 1} (-1)^k J_{2kT}(X) k^2 h(k)\right] \\&\ = \  c_1\sum_{k\geq 0} (-1)^k J_{2k+1}(X) (2k+1) w_T\left(\left(k+\frac{1}{2}\right)i\right) \nonumber\\&\quad\quad\quad\left[+ O\left(Xe^{-c_3 T}\right)\right],
\end{align}
where $c_1$, $c_2$, and $c_3$ are constants independent of $X$ and $T$, and the terms in brackets are included if and only if $w_T = h_T$.
\end{prop}

Since Bessel functions of integer order are much better-studied objects than those of purely imaginary order, this is a useful reduction. The calculation also realizes the Kloosterman term in the Kuznetsov formula as a sort of average (though the ``weight function'' is growing exponentially in the case of $H_T$) of Kloosterman terms arising in Petersson formulas over all even weights.

For what follows we quickly note the following corollary, which determines the size of the denominator (total mass) mentioned above.

\begin{cor}\label{cor:denominatorsize} Let $w_T$ equal $h_T$ or $H_T$ (as above). Then
\begin{align}\sum_{u\in \mathcal{B}_N} \frac{w_T(t_u)}{||u||^2}\ \asymp \  T^2\nu(N).\end{align}
\end{cor}

\begin{proof}
It suffices to show that \begin{align}\frac{\nu(N)}{N}\sum_{c\geq 1} \frac{S(1,1;c)}{c}\sum_{k\geq 0} (-1)^k J_{2k+1}\left(\frac{4\pi}{cN}\right) (2k+1) w_T\left(\left(k+\frac{1}{2}\right)i\right)\ \ll \  \frac{\nu(N) T}{N}.\end{align} To see this, we bound trivially by using \begin{align}J_k(x)\ \ll \  \frac{(x/2)^k}{k!},\end{align} and \begin{align}\sin\left(\frac{2k+1}{2T}\pi\right)\ \gg \  T^{-1}\end{align} in the case of $w_T = h_T$.
\end{proof}

\section{Proofs of the Main Theorems}\label{sec:proofsofmaintheorems}

Using the results from the previous section, we can now prove our main theorems. All arguments begin with the following reductions. We first use \eqref{eq:RforcedtobeTTN} to reduce the determination of the 1-level density to that of sums of the weighted averages of $\lambda_p$ and $\lambda_{p^2}$. We then use the Kuznetsov trace formula (Theorem \ref{thm:kuznetsovtraceformula}) to analyze these sums. By Lemma \ref{lem:kuznetsovreduction} we are reduced to bounding the contribution from the Bessel-Kloosterman term, to which we apply Proposition \ref{prop:usefulexpansionbesselterm} to analyze these exponential sums. We now turn to the details of each of these cases.

\subsection{Proof of Theorem \ref{maintheoremone}}

It suffices to study \begin{eqnarray}\mathcal{S} &\ := \  & \sum_{\ell=1}^2\sum_p \frac{2\log{p}}{p^{\frac{\ell}{2}}\log{R}}\widehat{\phi}\left(\frac{\ell\log{p}}{\log{R}}\right)\frac{\nu(N)}{N}\sum_{c\geq 1} \frac{S(p^\ell,1;c)}{c} \nonumber\\ & & \ \ \ \ \  \cdot \ \sum_{k\geq 0} (-1)^k J_{2k+1}\left(\frac{4\pi p^{\frac{\ell}{2}}}{cN}\right) (2k+1)w_T\left(\frac{2k+1}{2}i\right),\end{eqnarray} and bound $\mathcal{S}$ by something growing strictly slower than $\nu(N)$. This is because we get to divide this term by the total mass, which by Corollary \ref{cor:denominatorsize} is of the order $T^2\nu(N)$. As $T$ is fixed, we are dividing by a quantity on the order of $\nu(N)$.

Bounding trivially, we find
\begin{align}
\mathcal{S} &\ \ll \  \frac{\nu(N)}{N}\sum_{c\geq 1} \frac{1}{c}\sum_{k\geq 0} \frac{(2\pi)^{2k+1}w_T\left(\frac{2k+1}{2}i\right)}{(2k)!}\sum_{\ell=1}^2\sum_{p^{\ell/2}\leq R^{\eta/2}} \frac{2\log{p}}{\log{R}} (cN)^{\frac{1}{2}+\eps-2k-1} p^{k\ell}
\nonumber\\&\ \ll \  \frac{\nu(N)}{N^{\frac{3}{2}-\eps}}\sum_{k\geq 0} \frac{(2\pi)^{2k}}{(2k)!}w_T\left(\frac{2k+1}{2}i\right)\frac{R^{\eta(k+1)}}{N^{2k}\log{R}}
\nonumber\\&\ \ll \  \nu(N)\frac{N^{\eta - \frac{3}{2}+\eps} T^{2\eta}}{\log{N} + \log{T}} e^{O_T(T^{2\eta} N^{\eta-2})}.
\end{align} As $T$ is fixed, the above is negligible for $\eta < 3/2$, which completes the proof. \hfill $\Box$

\subsection{Proof of Theorem \ref{maintheoremtwo}}

It suffices to study \begin{eqnarray}\mathcal{S} & \ := \ & \sum_{\ell=1}^2\sum_p \frac{\log{p}}{p^{\frac{\ell}{2}}\log{T}}\widehat{\phi}\left(\frac{\ell\log{p}}{2\log{T}}\right)\sum_{c\geq 1} \frac{S(p^\ell,1;c)}{c}\nonumber\\ & & \ \ \ \ \ \ \cdot \ \sum_{k\geq 0} (-1)^k J_{2k+1}\left(\frac{4\pi p^{\frac{\ell}{2}}}{c}\right) (2k+1)h_T\left(\frac{2k+1}{2}i\right)\end{eqnarray} and bound $\mathcal{S}$ by something growing strictly slower than $T^2$. This is because $N$ is fixed, so by Corollary \ref{cor:denominatorsize} the denominator that occurs in the weighted averages is on the order of $T^2$.

In \cite{AM} (see their (3.8) to (3.18) --- for the convenience of the reader, this argument is reproduced in \S\ref{sec:appendixmanipulationfromAM} (Appendix II), it is proved that
\begin{align}
&\sum_{k\geq 0} (-1)^k J_{2k+1}\left(\frac{4\pi p^{\frac{\ell}{2}}}{c}\right) (2k+1)h_T\left(\frac{2k+1}{2}i\right) \nonumber\\&\quad\quad\quad= c_8T\sum_{|\alpha| < \frac{T}{2}} e\left(Y\sin\left(\frac{\pi\alpha}{T}\right)\right)\tilde{\tilde{h}}\left(\frac{\pi Y}{T}\cos\left(\frac{\pi\alpha}{T}\right)\right) + O(Y)\nonumber\\&\quad\quad\quad\ =: \  c_8TS_h(Y) + O(Y),
\end{align}
where $\tilde{\tilde{h}}(x):=x^2h(x)$ and $Y:=2p^{\ell/2}/c$. We apply the Euler-Maclaurin summation formula to the first term, yielding
\begin{align}\label{euler maclaurin}
S_h(Y) &\ = \  \int_{-\frac{T}{2}}^{\frac{T}{2}} e\left(Y\sin\left(\frac{\pi\alpha}{T}\right)\right)\tilde{\tilde{h}}\left(\frac{\pi Y}{T}\cos\left(\frac{\pi\alpha}{T}\right)\right) d\alpha \nonumber\\&\quad\quad+ \sum_{k=2}^M \frac{B_k}{k!}\left(e\left(Y\sin\left(\frac{\pi\alpha}{T}\right)\right)\tilde{\tilde{h}}\left(\frac{\pi Y}{T}\cos\left(\frac{\pi\alpha}{T}\right)\right)\right)^{(k)}\Bigg\vert_{-\frac{T}{2}}^{\frac{T}{2}} \nonumber\\&\quad\quad\quad\quad+ O\left(\int_{-\frac{T}{2}}^{\frac{T}{2}} \left|\left(e\left(Y\sin\left(\frac{\pi\alpha}{T}\right)\right)\tilde{\tilde{h}}\left(\frac{\pi Y}{T}\cos\left(\frac{\pi\alpha}{T}\right)\right)\right)^{(M)}\right| d\alpha\right).
\end{align}

In differentiating the expression \begin{align}e\left(Y\sin\left(\frac{\pi\alpha}{T}\right)\right)\tilde{\tilde{h}}\left(\frac{\pi Y}{T}\cos\left(\frac{\pi\alpha}{T}\right)\right)\end{align} $k$ times, the worst case is when we differentiate the exponential every single time and pick up a factor of $\left(\frac{Y}{T}\right)^k$; otherwise we gain at least one factor of $T$ (remember that $Y$ should be thought of as order $T^\eta$). Hence we may bound the error term by \begin{align}\int_{-\frac{T}{2}}^{\frac{T}{2}} \left|\left(e\left(Y\sin\left(\frac{\pi\alpha}{T}\right)\right)\tilde{\tilde{h}}\left(\frac{\pi Y}{T}\cos\left(\frac{\pi\alpha}{T}\right)\right)\right)^{(M)}\right| d\alpha\ \ll \  \left(\frac{Y}{T}\right)^M T.\end{align} Taking $M\geq 1 + \frac{1}{1-\eta}$, the error term is thus $O\left(Y/T\right)$.

Next, by the same analysis, in the second term of \eqref{euler maclaurin} we either differentiate the exponential every single time, or we gain a factor of $T$ from differentiating $\tilde{h}$ or one of the $\cos\left(\pi\alpha/T\right)$'s produced from differentiating the exponential. Thus, since we differentiate at least twice, all but one term in the $k$-fold derivative is bounded by $Y/T^2$. The last remaining term, obtained by differentiating the exponential $k$ times, vanishes because $\tilde{\tilde{h}}$ has a zero at $0$.

Hence it remains to bound the first term of \eqref{euler maclaurin}. This we do by integrating by parts, via
\begin{align}\int e^{\phi(x)}f(x) dx\ = \ -\frac{1}{2\pi i}\int e^{\phi(x)}\left(\frac{f(x)}{\phi'(x)}\right)' dx.\end{align} We get
\begin{align}
&\int_{-\frac{T}{2}}^{\frac{T}{2}} e\left(Y\sin\left(\frac{\pi\alpha}{T}\right)\right)\tilde{\tilde{h}}\left(\frac{\pi Y}{T}\cos\left(\frac{\pi\alpha}{T}\right)\right) d\alpha \nonumber\\&\quad\quad=\ c_{10}\frac{Y}{T^2}\int_{-\frac{T}{2}}^{\frac{T}{2}} e\left(Y\sin\left(\frac{\pi\alpha}{T}\right)\right)\tilde{h}'\left(\frac{\pi Y}{T}\cos\left(\frac{\pi\alpha}{T}\right)\right)\sin\left(\frac{\pi\alpha}{t}\right)d\alpha\nonumber\\&\quad\quad\ \ll\ \  \frac{Y}{T}.
\end{align}

Hence we obtain the bound \begin{align}S_h(Y)\ \ll \  \frac{Y}{T}.\end{align} Thus
\begin{align}\sum_{k\geq 0} (-1)^k J_{2k+1}\left(\frac{4\pi p^{\ell/2}}{c}\right) (2k+1)h_T\left(\frac{2k+1}{2}i\right)\ \ll \  \frac{p^{\ell/2}}{c}.\end{align}

That is, this tells us that
\begin{eqnarray}
& &\sum_{\ell=1}^2\sum_p \frac{\log{p}}{p^{\ell/2}\log{T}}\widehat{\phi}\left(\frac{\ell\log{p}}{2\log{T}}\right)\sum_{c\geq 1} \frac{S(p^\ell,1;c)}{c}\nonumber\\ & & \ \ \ \ \ \ \ \ \ \cdot\ \left[\sum_{k\geq 0} (-1)^k J_{2k+1}\left(\frac{4\pi p^{\ell/2}}{c}\right) (2k+1)h_T\left(\frac{2k+1}{2}i\right)\right]
\nonumber\\& & \ll \  \sum_{\ell=1}^2\sum_p \frac{\log{p}}{p^{\ell/2}\log{T}}\left|\widehat{\phi}\left(\frac{\ell\log{p}}{2\log{T}}\right)\right|\sum_{c\geq 1} c^{-\frac{1}{2}+\eps} \left[\frac{p^{\ell/2}}{c}\right]
\nonumber\\& & \ll \  \sum_{\ell=1}^2\sum_{p^{\ell/2}\leq T^\eta} \frac{\log{p}}{\log{T}}
\nonumber\\& & \ll \  \frac{T^{2\eta}}{\log{T}}.
\end{eqnarray}
For $\eta < 1$ this is negligible upon division by the total mass (which is of order $T^2$), completing the proof. \hfill $\Box$

\subsection{Proof of Theorem \ref{maintheoremthree}}

Before proceeding it bears repeating that the same limiting support as $K$ gets large (namely, $(-1,1)$) can be achieved by just trivially bounding as above (that is, without exploiting cancellation in sums of Kloosterman sums), but we present here an argument connecting the Kuznetsov formula to the Petersson formula as studied in \cite{ILS} instead.

It suffices to study \begin{eqnarray}\mathcal{S} & \ :=\ & \sum_{\ell=1}^2\sum_p \frac{\log{p}}{p^{\ell/2}\log{T}}\widehat{\phi}\left(\frac{\ell\log{p}}{2\log{T}}\right)\sum_{c\geq 1} \frac{S(p^\ell,1;c)}{c}\nonumber\\ & & \ \ \ \ \ \ \cdot \ \sum_{k\geq 0} (-1)^k J_{2k+1}\left(\frac{4\pi p^{\frac{\ell}{2}}}{c}\right) (2k+1)H_T\left(\frac{2k+1}{2}i\right)\end{eqnarray} and bound $\mathcal{S}$ by something growing strictly slower than $T^2$. Let \begin{align}Q_k^*(m;c)\ :=\ 2\pi i^k\sum_p S(p,1;c)J_{k-1}\left(\frac{4\pi m\sqrt{p}}{c}\right)\widehat{\phi}\left(\frac{\log{p}}{\log{R}}\right) \frac{2\log{p}}{\sqrt{p}\log{R}},\end{align} exactly as in \cite{ILS}. Then this simplifies to (dropping constants) \begin{align}\sum_{c\geq 1} c^{-1}\sum_{k\geq 0} (2k+1)H\left(\frac{2k+1}{2T}i\right)Q_{2k+2}^*(1;c)\end{align} if we ignore the $\ell = 2$ term, which is insignificant by e.g.\ GRH for symmetric square $L$-functions on $\GL_3/\Q$ (as in \cite{ILS}).

In Sections 6 and 7 of \cite{ILS} they prove
\begin{thm}
Assume GRH for all Dirichlet $L$-functions. Then \begin{align}Q_k^*(m;c)\ \ll \  \tilde{\gamma}_k(z) m T^\eta k^\eps (\log(2c))^{-2}\end{align} where \begin{align}\tilde{\gamma}_k(z)\ :=\ \begin{cases} 2^{-k} & k\geq 3z\\ k^{-1/2} & \text{otherwise}\end{cases}\end{align} and \begin{align}z\ :=\ \frac{4\pi T^\eta}{c}.\end{align}
\end{thm}

Hence the sum over $c$ now converges with no problem, and we may ignore it. What remains is
\begin{align}
&T^\eta \sum_{k\geq 0} (2k+1)H\left(\frac{2k+1}{2T}i\right)\tilde{\gamma}_k(z) k^\eps
\nonumber\\&\quad\quad\quad=\ T^\eta \left(\sum_{0\leq k\ \ll \  T^\eta} k^{\eps+\frac{1}{2}} H\left(\frac{2k+1}{2T}i\right) + \sum_{T^\eta\ \ll \  k} \frac{k^{1+\eps}H\left(\frac{2k+1}{2T}i\right)}{2^k}\right).
\end{align}

The second term in parentheses poses no problem. For the first term, using $H(x)\ll  x^K$, we see that the above is bounded by
\begin{align}
T^\eta\sum_{0\leq k\ \ll \  T^\eta} k^{\eps+\frac{1}{2}+K}T^{-K}\ \ll \  T^{\frac{5}{2}\eta-(1-\eta)K}.
\end{align}

Thus we need \begin{align}\eta\ <\ \frac{2+K}{\frac{5}{2}+K}\ = \ 1 - \frac{1}{5+2K}.\end{align} Again, by taking $K$ even larger we could have just trivially bounded throughout and not invoked \cite{ILS} or GRH, but the connection noted above may be of independent interest.

\subsection{Proof of Theorem \ref{maintheoremfour}}

By definition \begin{eqnarray} D_2(\mathcal{B}_1,u,\phi_1,\phi_2,R;w_T)& \ =\ & D_1(\mathcal{B}_1,u,\phi_1,R;w_T)D_1(\mathcal{B}_1,u,\phi_2,R;w_T)\nonumber\\ & & \ \ \  -\ 2D_1(\mathcal{B}_1,u,\phi_1\phi_2,R) + \delta_{(-1)^{\eps+\eps'+1},-1} \phi_1(0)\phi_2(0).\nonumber\\ \end{eqnarray} Averaging, we see that for $\widehat{\phi_1}$ and $\widehat{\phi_2}$ of sufficiently small support (actually $\left(-\frac{1}{2},\frac{1}{2}\right)$ would work fine), given our results above on one-level densities, up to negligible error \begin{eqnarray}D_2(\mathcal{B}_1,\phi_1,\phi_2,R;w_T) & \ :=\ & \Avg\left(D_1(u,\phi_1,R;w_T)D_1(u,\phi_2,R;w_T);\frac{w_T(t_u)}{||u||^2}\right)\nonumber\\ & & \ \ -\ (1-\mathcal{N}(-1))\phi_1(0)\phi_2(0) - 2\widehat{\phi_1}*\widehat{\phi_2}(0).\end{eqnarray} Now \begin{eqnarray} & & D_1(u,\phi_1,R;w_T)D_1(u,\phi_2,R;w_T) \ = \ \nonumber\\ & & \ \ \ \ \  \left(\widehat{\phi_1}(0)\frac{\log(1 + t_u^2)}{\log{R}} + \frac{\phi_1(0)}{2} - \sum_{\ell=1}^2\sum_p \frac{2\log{p}}{p^{\frac{\ell}{2}}\log{R}}\widehat{\phi_1} \left(\frac{\ell\log{p}}{\log{R}}\right)\right)\nonumber\\ & & \ \ \ \ \ \ \ \  \cdot\ \left(\widehat{\phi_2}(0)\frac{\log(1 + t_u^2)}{\log{R}} + \frac{\phi_2(0)}{2} - \sum_{\ell=1}^2\sum_p \frac{2\log{p}}{p^{\frac{\ell}{2}}\log{R}}\widehat{\phi_2}\left(\frac{\ell\log{p}}{\log{R}}\right)\right).\end{eqnarray}

Since $w_T$ is supported essentially around $t_u\asymp T$, the $\log(1+t_u^2)$ terms are all, up to negligible error, approximately $\log{R}$ (again we invoke the bound of \cite{Smi} on the $L^2$-norms occurring in the denominator). Also by an application of Kuznetsov, now with the inner product of two Hecke operators (namely $T_{p^\ell}$ and $T_{q^{\ell'}}$ for $p,q$ primes and $0\leq \ell,\ell'\leq 2$ --- this uses the results on the Bessel-Kloosterman term established above), we see that the resulting average is, up to negligible error,
\begin{eqnarray}
& & \Avg\left(D_1(u,\phi_1,R)D_1(\phi_2,u);\frac{w_T(t_u)}{||u||^2}\right)\ =\ \nonumber\\ & & \left(\widehat{\phi_1}(0) + \frac{\phi_1(0)}{2}\right)\left(\widehat{\phi_2}(0) + \frac{\phi_2(0)}{2}\right) + \sum_{\ell=1}^2\sum_p \frac{4\log^2{p}}{p^\ell\log^2{R}}\widehat{\phi_1}\left(\frac{\ell\log{p}}{\log{R}}\right)\widehat{\phi_2}\left(\frac{\ell\log{p}}{\log{R}}\right).\nonumber\\
\end{eqnarray}
That is, only the diagonal terms $p^\ell=q^{\ell'}$ matter. Now partial summation (and the prime number theorem, as usual) finishes the calculation.


\section{Appendix I: Contour integration}\label{sec:appendixAcontourintegration}

We prove Proposition \ref{prop:usefulexpansionbesselterm} below. We restate it for the reader's convenience.

\begin{prop} Let $T$ be an odd integer and $X\leq T$. Let $w_T$ equal $h_T$ or $H_T$, where these are the weight functions from Theorems \ref{maintheoremone} to \ref{maintheoremfour}. Then \begin{align}
\int_\R J_{2ir}(X)\frac{rw_T(r)}{\cosh(\pi r)}dr &\ = \  c_1\sum_{k\geq 0} (-1)^k J_{2k+1}(X) (2k+1) w_T\left(\left(k+\frac{1}{2}\right)i\right) \nonumber\\&\quad\quad\quad\left[ + \ c_2 T^2 \sum_{k\geq 1} (-1)^k J_{2kT}(X) k^2 h(k)\right] \\&\ = \  c_1\sum_{k\geq 0} (-1)^k J_{2k+1}(X) (2k+1) w_T\left(\left(k+\frac{1}{2}\right)i\right) \nonumber\\&\quad\quad\quad\left[+ O\left(Xe^{-c_3 T}\right)\right],
\end{align}
where $c_1$, $c_2$, and $c_3$ are constants independent of $X$ and $T$, and the terms in brackets are included if and only if $w_T = h_T$.
\end{prop}

\begin{proof} In the proof below bracketed terms are present if and only if $w_T = h_T$.

Recall that \begin{align}J_\alpha(2x)\ =\ \sum_{m\geq 0} \frac{(-1)^m x^{2m+\alpha}}{m!\Gamma(m+\alpha+1)}.\end{align} By Stirling's formula, $\Gamma(m+\alpha+1)\cosh(\pi r)\gg |m+2ir+1|^{m+1}$. Hence by the Lebesgue Dominated Convergence Theorem (remember that $w_T(z)$ is of rapid decay as $|\Re{z}|\to\infty$) we may switch sum and integral to get
\begin{align}\label{appendix a step one}
&\int_\R J_{2ir}(X)\frac{rw_T(r)}{\cosh(\pi r)}dr \nonumber\\&\quad\quad \ =\ \sum_{m\geq 0}\frac{(-1)^m x^{2m}}{m!}\int_\R \frac{x^{2ir} rw_T(r)}{\Gamma(m+1+2ir)\cosh(\pi r)} dr,
\end{align}
where $X=:2x$.

Now we move the line of integration down to $\R - iR$, $R\not\in \Z + \frac{1}{2}$ (or $T\Z$ if $w_T = h_T$). To do this, we note the estimate (for $A\gg 1$)
\begin{align}
\int_{\pm A\to \pm A-iR} \frac{x^{2ir} rw_T(r)}{\Gamma(m+1+2ir)\cosh(\pi r)} dr&\ll \int_{\pm A\to \pm A-iR} x^{2R} |r| |m+2ir|^{-m} |w_T(r)| dr\nonumber\\&\ll _{T,R} A^{-2},
\end{align}
where again we have used the rapid decay of $w_T$ along horizontal lines, and $B\to B-iR$ denotes the vertical line from $B\in \C$ to $B-iR\in \C$. (Rapid decay also ensures the integral along $\R - iR$ converges absolutely.)

Note that the integrand \begin{align}\frac{x^{2ir} rw_T(r)}{\Gamma(m+1+2ir)\cosh(\pi r)}\end{align} has poles precisely at $r\in -i(\N + \frac{1}{2}) = \{-\frac{1}{2}i,-\frac{3}{2}i,\ldots\}$, and, if $w_T = h_T$, poles also at $r\in -iT\Z^+$. The residue of the pole at $r = -\frac{2k+1}{2}i$ ($k\geq 0$) is, up to an overall constant independent of $k$, \begin{align}\frac{x^{2k+1}}{\Gamma(m+1+(2k+1))} (-1)^k(2k+1)w_T\left(\frac{2k+1}{2}i\right).\end{align} If $w_T = h_T$, the residue of the pole at $r = -ikT$ ($k\geq 1$) is, up to another overall constant independent of $k$, \begin{align}\frac{x^{2kT}}{\Gamma(m+1+(2kT))} k^2T^2 \frac{h(k)}{\cos(\pi kT)} = \frac{x^{2kT}}{\Gamma(m+1+(2kT))} (-1)^k k^2T^2h(k).\end{align}

Hence the sum of \eqref{appendix a step one} becomes
\begin{align}
&\sum_{m\geq 0}\frac{(-1)^m x^{2m}}{m!}\int_\R \frac{x^{2ir} rw_T(r)}{\Gamma(m+1+2ir)\cosh(\pi r)} dr\nonumber\\&\quad\quad=\ c_1\sum_{m\geq 0} \frac{(-1)^m x^{2m}}{m!}\left(\sum_{0\leq k\ll R} \frac{x^{2k+1}}{\Gamma(m+1+(2k+1))} (-1)^k(2k+1)w_T\left(\frac{2k+1}{2}i\right) \right.\nonumber\\&\quad\quad\quad\quad\quad\left.+ \int_{\R - iR} \frac{x^{2ir} rw_T(r)}{\Gamma(m+1+2ir)\cosh(\pi r)} dr\right.\nonumber\\&\quad\quad\quad\quad\quad\quad\quad\left. \left[+ c_2\sum_{m\geq 0} \frac{(-1)^m x^{2m}}{m!}\sum_{0\leq k\ll R} \frac{x^{2kT}}{\Gamma(m+1+(2kT))} (-1)^k k^2 T^2h(k)\right]\right).
\end{align}

Now we take $R\to\infty$. Note that
\begin{align}
\int_{\R - iR} \frac{x^{2ir} rw_T(r)}{\Gamma(m+1+2ir)\cosh(\pi r)} dr&\ \ll\ \int_{\R - iR} x^{2R} |r||w_T(r)||m+2ir+1|^{-m-1-2R} dr\nonumber\\&\ \ll_{m}\ x^{2R} e^{\frac{\pi R}{2T}} R^{-m-1-2R},
\end{align}
again by Stirling and rapid decay of $w_T$ on horizontal lines (that is, $w_T(x+iy)\ll (1+x)^{-4}e^{\frac{\pi y}{2T}}$ since both $h$ and $H$ have all their derivatives supported in $\left(-\frac{1}{4},\frac{1}{4}\right)$). This of course vanishes as $R\to\infty$.

Hence we see that
\begin{eqnarray}
& & \int_\R J_{2ir}(X)\frac{rw_T(r)}{\cosh(\pi r)}dr \nonumber\\ & & \ \ \ \ \ \ =\  c_1\sum_{m\geq 0} \frac{(-1)^m x^{2m}}{m!}\sum_{k\geq 0} \frac{x^{2k+1}}{\Gamma(m+1+(2k+1))} (-1)^k(2k+1)w_T\left(\frac{2k+1}{2}i\right)\nonumber\\& & \ \ \ \ \ \ \ \ \ \ \ \left[+ c_2\sum_{m\geq 0} \frac{(-1)^m x^{2m}}{m!}\sum_{k\geq 0} \frac{x^{2kT}}{\Gamma(m+1+(2kT))} (-1)^k k^2 T^2h(k)\right].
\end{eqnarray}

Switching sums (via the exponential bounds on $w_T$ along the imaginary axis) and applying $J_n(X) = \sum_{m\geq 0} \frac{(-1)^m x^{2m+n}}{m!(m+n)!}$ gives us the claimed calculation. For the bound on the bracketed term, use \begin{align}J_n(ny)\ \ll\ \frac{y^n e^{n\sqrt{1-y^2}}}{(1+\sqrt{1-y^2})^n}\end{align} (see \cite{AS}, page 362) and bound trivially.
\end{proof}

\section{Appendix II: An exponential sum identity}\label{sec:appendixmanipulationfromAM}

The following proposition and proof are also used in \cite{AM}.

\begin{prop}
Suppose $X\leq T$. Then
\begin{align}\label{actual small bound}
S_J(X)\ &:=\ T\sum_{k\geq 0} (-1)^k J_{2k+1}(X) \frac{\tilde{\tilde{h}}\left(\frac{2k+1}{2T}\right)}{\sin\left(\frac{2k+1}{2T}\pi\right)}\nonumber\\&\ =\ c_8T\sum_{|\alpha| < \frac{T}{2}} e\left(Y\sin\left(\frac{\pi\alpha}{T}\right)\right) \tilde{\tilde{g}}\left(\frac{\pi Y}{T}\cos\left(\frac{\pi\alpha}{T}\right)\right) + O(Y),
\end{align}
where $c_8$ is some constant, $X =: 2\pi Y$, and for any $f$ set $\tilde{f}(x):=xf(x)$.
\end{prop}

\begin{proof}
Observe that $k\mapsto \sin\left(\pi k/2\right)$ is supported only on the odd integers, and maps $2k+1$ to $(-1)^k$. Hence, rewriting gives
\begin{align}\label{actual two}
S_J(X)\ =\ T\sum_{k\geq 0 \atop k\not\in 2T\Z} J_k(X) \tilde{\tilde{h}}\left(\frac{k}{2T}\right)\frac{\sin\left(\frac{\pi k}{2}\right)}{\sin\left(\frac{\pi k}{2T}\right)}.
\end{align}
As
\begin{align}
\frac{\sin\left(\frac{\pi k}{2}\right)}{\sin\left(\frac{\pi k}{2T}\right)}\ =\ \frac{e^{\frac{\pi i k}{2}} - e^{-\frac{\pi i k}{2}}}{e^{\frac{\pi i k}{2T}} - e^{\frac{\pi i k}{2T}}}\ =\ \sum_{\alpha = -\left(\frac{T-1}{2}\right)}^{\frac{T-1}{2}} e^{\frac{\pi i k \alpha}{T}}
\end{align}
when $k$ is not a multiple of $2T$, we find that
\begin{align}\label{actual three}
S_J(X)\ =\ T\sum_{|\alpha| < \frac{T}{2}} \sum_{k\geq 0 \atop k\not\in 2T\Z} e\left(\frac{k\alpha}{2T}\right) J_k(X) \tilde{\tilde{h}}\left(\frac{k}{2T}\right).
\end{align}
Observe that, since the sum over $\alpha$ is invariant under $\alpha\mapsto -\alpha$ (and it is non-zero only for $k$ odd!), we may extend the sum over $k$ to the entirety of $\Z$ at the cost of a factor of 2 and of replacing $h$ by \begin{equation}\label{definition of g}g(x)\ := \ \sgn(x)h(x).\end{equation} Note that $g$ is as differentiable as $h$ has zeros at $0$, less one. That is to say, $\widehat{g}$ decays like the reciprocal of a degree $\ord_{z=0} h(z) - 1$ polynomial at $\infty$. This will be crucial in what follows.

Next, we add back on the $2T\Z$ terms and obtain
\begin{eqnarray}\label{actual four}
\frac{1}{2}S_J(X) &\ =\ & T\sum_{|\alpha| < \frac{T}{2}} \sum_{k\in \Z} e\left(\frac{k\alpha}{2T}\right) J_k(X) \tilde{\tilde{g}}\left(\frac{k}{2T}\right) - T^2\sum_{k\in \Z} J_{2kT}(X) k^2 h(k)
\nonumber\\ &=& T\sum_{|\alpha| < \frac{T}{2}} \sum_{k\in \Z} e\left(\frac{k\alpha}{2T}\right) J_k(X) \tilde{\tilde{g}}\left(\frac{k}{2T}\right) + O\left(Xe^{-c_4T}\right)\nonumber\\ & =: & V_J(X) + O\left(Xe^{-c_4T}\right),
\end{eqnarray}
where we have bounded the term $T^2\sum_{k\in \Z} J_{2kT}(X) k^2 h(k)$ trivially via $J_n(2x)\ll x^n/n!$.

Now we move to apply Poisson summation. Write $X=:2\pi Y$. We apply the integral formula (for $k\in \Z$) \begin{equation}J_k(2\pi x)\ =\ \int_{-\frac{1}{2}}^{\frac{1}{2}} e\left(kt - x\sin(2\pi t)\right) dt\end{equation} and interchange the sum and integral (via rapid decay of $g$) to get that
\begin{align}\label{actual five}
V_J(X)\ =\ T\sum_{|\alpha| < \frac{T}{2}} \int_{-\frac{1}{2}}^{\frac{1}{2}} \left(\sum_{k\in \Z} e\left(\frac{k\alpha}{2T} + kt\right) \tilde{\tilde{g}}\left(\frac{k}{2T}\right)\right) e\left(-Y\sin(2\pi t)\right) dt.
\end{align}
By Poisson summation, \eqref{actual five} is just (interchanging the sum and integral once more)
\begin{eqnarray}\label{actual six}
V_J(X) &\ =\ & T^2\sum_{|\alpha| < \frac{T}{2}} \sum_{k\in \Z}\int_{-\frac{1}{2}}^{\frac{1}{2}} \widehat{g}''\left(2T(t-k) + \alpha\right) e\left(-Y\sin(2\pi t)\right) dt \nonumber\\&=& c_5T\sum_{|\alpha| < \frac{T}{2}} \int_{-\infty}^\infty \widehat{g}''(t) e\left(Y\sin\left(\frac{\pi t}{T} + \frac{\pi \alpha}{T}\right)\right) dt\nonumber\\&=:&c_5W_g(X).
\end{eqnarray}

As \begin{equation}\sin\left(\frac{\pi t}{T} + \frac{\pi\alpha}{T}\right)\ =\ \sin\left(\frac{\pi\alpha}{T}\right) + \frac{\pi t}{T}\cos\left(\frac{\pi\alpha}{T}\right) - \frac{\pi^2 t^2}{T^2}\sin\left(\frac{\pi\alpha}{T}\right) + O\left(\frac{t^3}{T^3}\right),\end{equation} we see that
\begin{eqnarray}
& & W_g(X) \nonumber\\ & = \ & c_6T\sum_{|\alpha| < \frac{T}{2}} e\left(Y\sin\left(\frac{\pi\alpha}{T}\right)\right) \int_{-\infty}^\infty \widehat{g}''(t) e\left(\frac{\pi Y t}{T}\cos\left(\frac{\pi\alpha}{T}\right)\right) dt \nonumber\\ & & \ \ -\ c_7\frac{\pi^2 Y}{T}\sum_{|\alpha| < \frac{T}{2}} e\left(Y\sin\left(\frac{\pi\alpha}{T}\right)\right) \sin\left(\frac{\pi\alpha}{T}\right) \int_{-\infty}^\infty t^2\widehat{g}''(t) e\left(\frac{\pi Y t}{T}\cos\left(\frac{\pi\alpha}{T}\right)\right) dt \nonumber\\& & \ \ +\ O\left(\frac{Y}{T} + \frac{Y^2}{T^2}\right)
\label{actual seven}\\ &  = & c_8T\sum_{|\alpha| < \frac{T}{2}} e\left(Y\sin\left(\frac{\pi\alpha}{T}\right)\right) \tilde{\tilde{g}}\left(\frac{\pi Y}{T}\cos\left(\frac{\pi\alpha}{T}\right)\right) \nonumber\\& & \ \ -\ c_9\frac{Y}{T}\sum_{|\alpha| < \frac{T}{2}} e\left(Y\sin\left(\frac{\pi\alpha}{T}\right)\right) \sin\left(\frac{\pi\alpha}{T}\right) \tilde{\tilde{g}}''\left(\frac{\pi Y}{T}\cos\left(\frac{\pi\alpha}{T}\right)\right) \nonumber\\& &\ \ +\ O\left(\frac{Y}{T} + \frac{Y^2}{T^2}\right).
\end{eqnarray}
Now bound the second term trivially to get the claim.
\end{proof}

\ \\

\section*{Acknowledgements}

We thank Gergely Harcos, Andrew Knightly, Stephen D. Miller and Peter Sarnak for helpful conversations on an earlier version.  This work was done at the SMALL REU at Williams College, funded by NSF GRANT DMS0850577 and Williams College; it is a pleasure to thank them for their support. The second named author was also partially supported by the Mathematics Department of University College London, and the fifth named author was partially supported by NSF grants DMS0970067 and DMS1265673.

\ \\


\end{document}